\documentclass[12pt]{amsart}

\usepackage{amsthm, amssymb}
\usepackage{enumerate, mathrsfs}

\makeatletter
\@namedef{subjclassname@2010}{%
\textup{2010} Mathematics Subject Classification}
\makeatother

\allowdisplaybreaks

\theoremstyle{plain}
\newtheorem{thm}{Theorem}[section]

\newtheorem{lem}[thm]{Lemma}
\newtheorem{prop}[thm]{Proposition}
\newtheorem{cor}[thm]{Corollary}
\newtheorem{rem}[thm]{Remark}

\title[Optimal extension of the Fourier transform and convolution operator]{Optimal extension of the Fourier transform and convolution operator on compact groups}
\author{Manoj Kumar}
\address{Department of Mathematics \endgraf Indian Institute of Technology Delhi \endgraf Delhi - 110016 \endgraf India}
\email{manojk9t3@gmail.com}
\author{N. Shravan Kumar}
\address{Department of Mathematics \endgraf Indian Institute of Technology Delhi \endgraf Delhi - 110016 \endgraf India}
\email{shravankumar@maths.iitd.ac.in}

\begin{document}

\begin{abstract}
Let $G$ be a compact group (not necessarily abelian) and let $\Phi$ be a Young function satisfying the $\Delta_2$-condition. We determine the optimal domain and the associated extended operator for both Fourier transform and the convolution operator defined on the Orlicz spaces $L^\Phi(G).$
\end{abstract}


\maketitle

\section{Introduction}
Let $X$ be a $\sigma$-order continuous Banach function space \cite{ORP} defined over a finite positive measure space $(\Omega,\mathcal{A},\mu).$ If $Y$ is any Banach space and $T:X\rightarrow Y$ is a bounded linear operator, define $\nu_T:\mathcal{A}\rightarrow Y$ as $\nu_T(A)=T(\chi_A).$ It is well known that $\nu_T$ is a $\sigma$-additive vector measure. Further, by \cite[Pg. 194, Theorem 4.14]{ORP}, $X$ can be embedded continuously inside $L^1(\nu_T)$ and the mapping $I_{\nu_T} :L^1(\nu_T)\rightarrow Y$ given by $I_{\nu_T}(f)=\int_{\Omega}f\ d\nu_T$ is a continuous extension of $T.$ The mapping $I_{\nu_T}$ is called as the integration operator. Also, the space $L^1(\nu_T)$ is optimal in the sense that if $T$ extends continuously to a space $\widetilde{X}$ with values in $Y,$ then $\widetilde{X}$ can be embedded continuously inside $L^1(\nu_T).$

The determination of the optimal domain and the extension for operators arising from analysis is one of the classical problems of functional analysis. This optimal extension process is known for kernel operators \cite{CR1} and Sobolev embeddings \cite{CR2,EKP,KP}. 

In 2007, S. Okada and W. J. Ricker considered the optimal extension problem for the convolution operators arising from measures defined on compact abelian groups. Further, in \cite[Chapter 7]{ORP}, the authors have also studied the Fourier transform operator on compact abelian groups from the optimal domain view point. This paper aims to study the optimal domain, extension and their properties for both convolution operator and Fourier transform defined on a compact group. The initial domain for these operators are the Orlicz spaces on a compact group. In this way, these results are new even when the underlying group is abelian. 

In section 3, we study the Fourier transform operators while in section 4, we study the convolution operator and the compactness properties of the extended operator. 

\section{Preliminaries}

\subsection{Fourier analysis on compact groups}
Let $G$ be a compact Hausdorff group and let $m_G$ denote the normalized positive Haar measure on $G.$ For $1\leq p\leq\infty$,  $L^p(G)$ will denote the usual $p^{\mbox{th}}$-Lebesgue space with respect to the measure $m_G.$ It is well known that an irreducible unitary representation of a compact group $G$ is always finite-dimensional. Let $\widehat{G}$ be the set of all unitary equivalence classes of irreducible unitary representations of $G$. The set $\widehat{G}$ is called the unitary dual of $G$ and $\widehat{G}$ is given the discrete topology. 

Let $\{(X_\alpha,\|.\|_\alpha)\}_{\alpha\in\wedge}$ be a collection of Banach spaces. We shall denote by $\ell^1\mbox{-}\underset{\alpha\in\wedge}{\oplus}X_\alpha,$ the Banach space $\left\{(x_\alpha)\in\underset{\alpha\in\wedge}{\Pi}X_\alpha:\underset{\alpha\in\wedge}{\sum}\|x_\alpha\|_\alpha<\infty\right\}$ equipped with the norm $\|(x_\alpha)\|_\infty:=\underset{\alpha\in\wedge}{\sum}\|x_\alpha\|_\alpha.$ Similarly, we shall denote by $\ell^\infty\mbox{-}\underset{\alpha\in\wedge}{\oplus}X_\alpha,$ the Banach space $\left\{(x_\alpha)\in\underset{\alpha\in\wedge}{\Pi}X_\alpha:\underset{\alpha\in\wedge}{\sup}\|x_\alpha\|_\alpha<\infty\right\}$ equipped with the norm $\|(x_\alpha)\|_\infty:=\underset{\alpha\in\wedge}{\sup}\|x_\alpha\|_\alpha.$ In the same lines, we shall also denote by $c_0\mbox{-}\underset{\alpha\in\wedge}{\oplus}X_\alpha,$ the space consisting of those vectors $(x_\alpha)$ from $\ell^\infty\mbox{-}\underset{\alpha\in\wedge}{\oplus}X_\alpha$ which goes to $0$ as $\alpha\rightarrow\infty.$ It is clear that $c_0\mbox{-}\underset{\alpha\in\wedge}{\oplus}X_\alpha$ is a closed subspace of $\ell^\infty\mbox{-}\underset{\alpha\in\wedge}{\oplus}X_\alpha.$

Let $\pi$ be an irreducible unitary representation of $G$ on the Hilbert space $\mathcal{H}_\pi$ of dimension $d_\pi.$ For $f\in L^1(G),$ the Fourier transform of $f$ at $\pi$ denoted $\widehat{f}(\pi),$ is defined as $$\widehat{f}(\pi)=\int_Gf(t)\pi(t)^* \,dm_G(t),~[\pi]\in\widehat{G}.$$ Note that the Fourier transform operator $f\mapsto \widehat{f}$ maps $L^1(G)$ into $\ell^\infty\mbox{-}\underset{[\pi]\in\widehat{G}}{\oplus}\mathcal{B}_2(\mathcal{H}_\pi).$ This operator is injective and bounded. 
 
 For more details on compact groups, we refer to \cite{F}.

\subsection{Orlicz Spaces}
Let $ \Phi: [0,\infty]\rightarrow [0,\infty] $ be a convex function. Then $\Phi$ is called a Young function if it satisfies $ \Phi(0)= 0 $ and $\underset{x\rightarrow \infty}{\lim}  \Phi(x)= + \infty $. If $ \Phi$ is any Young function, then define $\Psi$ as
$$ \Psi(y):= \sup{\{xy-\Phi(x) : x\geq0\}} , y\geq 0.$$ Then $\Psi$ is also a Young function and is termed as the complementary function to $ \Phi.$ Further, the pair $ (\Phi,\Psi) $ is called a complementary pair of Young functions. A pair $(\Phi,\Psi)$ of complementary Young functions is said to be normalized if $\Phi(1)+\Psi(1)=1.$

We say that a  Young function $ \Phi $  satisfies  the $\Delta_{2}$-condition if there exists a constant $ K>0 $ and $ x_{0} > 0$ such that $\Phi(2x)\leq K\Phi(x)$ whenever $x\geq x_{0}.$

The Orlicz space, denoted $ L^{\Phi}(G),$ is a Banach space consisting of complex measurable functions on $G$ for which 
$$\|f\|_{\Phi} := \inf \left\{ k>0 :\int_G\Phi\left(\frac{|f|}{k}\right) dm_G \leq1 \right\}<\infty.$$ The above norm is called as the Luxemburg norm or Gauge norm.  If $(\Phi,\Psi)$ is a complementary Young pair, then there is a norm on $L^\Phi(G),$ equivalent to the Luxemburg norm, given by, $$ \|f\|_{\Phi}^o =\sup \Bigg \{ \int_{G}|fg|dm_G : \int_{G}\Psi(|g|)dm_G\leq1 \Bigg\}.$$ This norm is called as the Orlicz norm. Note that the Orlicz space is a natural generalisation of the classical $L^p$ spaces.

Let $\Phi$ be a Young function satisfying the $\Delta_2$-condition. Then the space of all continuous functions on $G$ is dense in $L^\Phi(G).$ Similarly, the space of simple functions is also dense in $L^\Phi(G).$ Further, the dual of $ (L^{\Phi}(G),\|\cdot\|_{\Phi}) $ is isometrically isomorphic to $ (L^{\Psi}(G),\|\cdot\|_{\Psi}^o).$ In particular, if both $\Phi$ and $\Psi$ satisfies the $\Delta_2$-condition, then $L^\Phi(G)$ is reflexive \cite[Pg. 112, Theorem 10]{RR}. Further, by \cite[Pg. 595, Theorem 133.3]{Z}, $L^\Phi(G)$ is a $\sigma$-order continuous Banach function space.

The following is an analogue of \cite[Pg. 380, Lemma 35.11]{HR}. The proof of this follows as in the case of $L^p$ spaces.
\begin{lem}\label{RNDOSF}
Let $\Phi$ be a Young function satisfying the $\Delta_2$-condition and let $\mu\in M(G)$ such that $$\sup\{\|\mu*f\|_\Phi:\|f\|_1\leq 1\}<\infty.$$ Then there exists $g\in L^\Phi(G)$ such that $d\mu=gdm_G.$
\end{lem}

For more on Orlicz spaces see \cite{RR}.

\subsection{Vector measure}
Let $G$ be a compact group. Let $\mathcal{B}(G)$ denote the $\sigma$-algebra consisting of all Borel subsets of $G.$ Further, let $M(G)$ denote the space of all bounded complex Radon measures on $G.$ Let $X$ be a complex Banach space and let $\nu$ be a $\sigma$-additive $X$-valued measure called vector measure on $G.$ Let $X^\prime$ be the dual of $X$ and let $B_{X^\prime}$ be the closed unit ball in $X^\prime.$ For each $x^\prime\in X^\prime,$ we shall denote by $\langle\nu,x^\prime\rangle,$ the corresponding scalar valued measure for the vector measure $\nu,$ which is defined as $\langle\nu,x^\prime\rangle(A)=\langle\nu(A),x^\prime\rangle,\  A\in\mathcal{B}(G).$ A set $A\in\mathcal{B}(G)$ is said to be $\nu$-null if $\nu(B)=0$ for every Borel set $B\subset A.$ The variation of $\nu,$ denoted $|\nu|,$ is a positive measure defined as follows: For a set $A\in\mathcal{B}(G),$ $$|\nu|(A)=\sup\left\{\sum_{E\in\rho}\|\nu(E)\|: \rho\mbox{ the finite partition of A}\right\}.$$ The vector measure $\nu$ is said to be measure of bounded variation if $|\nu|(G)<\infty.$ The semivariation of $\nu$ on a set $A\in\mathcal{B}(G)$ is given by $\|\nu\|(A)=\underset{x^\prime\in B_{X^\prime}}\sup |\langle\nu,x^\prime\rangle|(A),$ where $|\langle\nu,x^\prime\rangle|$ is the total variation of the scalar measure $\langle\nu,x^\prime\rangle.$ Let $\|\nu\|$ denote the quantity $\|\nu\|(G).$ The vector measure $\nu$ is said to be {\it absolutely continuous with respect to a non-negative scalar measure $\mu$} if $\underset{\mu(A)\rightarrow 0}{\lim}\nu(A)=0,\ A\in\mathcal{B}(G).$ A Banach space $X$ is said to have the {\it Radon-Nikodym Property} with respect to $(G,\mathfrak{B}(G),m_G)$ if for each $X$-valued absolutely continuos vector measure $\nu$ of bounded variation there exists a Bochner integrable function $f:G\rightarrow X$ such that $d\nu=f\,dm_G.$

A complex valued function $f$ on $G$ is said to be $\nu$-weakly integrable if $f\in L^1(|\langle\nu,x^\prime\rangle|),$ for all $x^\prime\in X^\prime.$ We shall denote by $L^1_w(\nu)$ the Banach space of all $\nu$-weakly integrable functions equipped with the norm $$\|f\|_\nu=\underset{x^\prime\in B_{X^\prime}}{\sup}\int_G |f|\,d|\langle\nu,x^\prime\rangle|.$$ A $\nu$-weakly integrable function $f$ is said to be $\nu$-integrable if for each $A\in\mathcal{B}(G)$ there exists a unique $x_A\in X$ such that $\int_A f\,d\langle\nu,x^\prime\rangle=\langle x_A,x^\prime\rangle,\ x^\prime\in X^\prime.$ The vector $x_A$ is denoted by $\int_Af\,d\nu.$ We shall denote by $L^1(\nu)$ the space of all $\nu$-integrable functions and it is also a Banach space when equipped with the $\|\cdot\|_\nu$ norm. We denote by $I_\nu$ the continuous linear operator $I_\nu:L^1(\nu)\rightarrow X$ given by $I_\nu(f)=\int_Gf\,d\nu,\ f\in L^1(\nu)$. See \cite[Pg. 152]{ORP}. Note that the space of simple functions on $G$ is dense in $L^1(\nu).$

Let $T:L^\Phi(G)\rightarrow X$ be a bounded linear operator. Define $\nu_T:\mathcal{B}(G)\rightarrow X$ as $\nu_T(A)=T(\chi_A).$ Then $\nu_T$ is a vector measure on $G.$ The measure $\nu_T$ will be called as the vector measure associated with the operator $T.$ See \cite[Pg. 185]{ORP}.

For more details on vector measures and integration with respect to vector measures, we refer to \cite{DU, D, ORP}.

Throughout this paper $G$ will be a compact group and $m_G$ will denote the unique normalized Haar measure on $G.$ Further, $\Phi$ will always denote a Young function satisfying the $\Delta_2$-condition and $\Psi$ will be the Young function which is complementary to $\Phi.$ Also, $L^\Phi(G)$ will be considered only with the Luxemburg norm while its dual $L^\Psi(G)$ will be considered only with the Orlicz norm.

\section{Optimal extension of the Fourier transform}

In this section, we study the Fourier transform from two different view points. Firstly we restrict the Fourier transform operator on $L^1(G)$ to $L^\Phi(G)$ by treating $L^\Phi(G)$ as a subspace of $L^1(G).$ On the other hand, there is also a Fourier transform on the $L^\Phi(G)$ spaces arising from the Hausdorff-Young inequality. The optimal domain for these operators is studied in this section. The results obtained in this section are motivated by the results presented in \cite[Chapter 7]{ORP}.

\subsection{Optimal extension of $c_0$-valued Fourier transform}
Define $\mathscr{F}_{\Phi,0}:L^\Phi(G)\rightarrow c_0\mbox{-}\underset{[\pi]\in\widehat{G}}{\oplus}\mathcal{B}_2(\mathcal{H}_\pi)$ as $\mathscr{F}_{\Phi,0}(f)=\widehat{f}.$ Here $\widehat{f}$ denotes the Fourier transform of $f$ as an element of $L^1(G),$ since $L^\Phi(G)$ is contained in $L^1(G).$ Further, it is well known that $\mathscr{F}_{\Phi,0}$ is a bounded linear operator. We shall denote by $\nu_{\Phi,0}$ the associated vector measure.

\begin{lem}\label{Variation_of_nu}
The vector measure $\nu_{\Phi,0}$ has finite variation and the variation of $\nu_{\Phi,0}$ coincides with the Haar measure $m_G.$
\end{lem}
\begin{proof}
Let $A\in\mathcal{B}(G).$ Then,
\begin{align*}
\|\nu_{\Phi,0}(A)\|_{\infty}=&\|\widehat{\chi_A}\|_{\infty}\leq \|\chi_A\|_1=m_G(A).
\end{align*}
Thus $|\nu_{\Phi,0}|$ is finite. Let $\pi_0$ denote the one dimensional representation given by $\pi_0(t)=1$ for all $t\in G.$ Let $$\chi_{\{\pi_0\}}(\pi)=\left\{\begin{array}{ccc}
1 & \mbox{if} & [\pi]=[\pi_0]\\
0_{d_\pi} & \mbox{if} & [\pi]\neq[\pi_0]
\end{array}\right.,$$ for $[\pi]\in\widehat{G}.$ Here $0_{d_\pi}$ denotes the $d_\pi\times d_\pi$ zero matrix. Then, it is clear that $\chi_{\{\pi_0\}}\in B_{\ell^1\mbox{-}\underset{[\pi]\in\widehat{G}}{\oplus}\mathcal{B}_2(\mathcal{H}_\pi)}.$ Further, \begin{align*}
\langle \nu_{\Phi,0}(A),\chi_{\{\pi_0\}} \rangle =& \langle \widehat{\chi_A},\chi_{\{\pi_0\}} \rangle \\ =& \underset{[\pi]\in\widehat{G}}{\sum} tr(\widehat{\chi_A}(\pi)\chi_{\{\pi_0\}}(\pi)) \\ =& tr(\widehat{\chi_A}(\pi_0)) = \widehat{\chi_A}(\pi_0) \\ =& \int_A\pi_0(t^{-1})\ dm_G(t) = m_G(A).
\end{align*}
Therefore, $m_G(A)=|\langle \nu_{\Phi,0}(A),\chi_{\{\pi_0\}} \rangle| \leq \|\nu_{\Phi,0}(A)\|_{\infty}.$ Hence $\|\nu_{\Phi,0}(A)\|_{\infty}=m_G(A).$ Since $A\in\mathcal{B}(G)$ is arbitrary, it follows that $|\nu_{\Phi,0}|=m_G.$
\end{proof}
As an immediate corollary we have the following.
\begin{cor}\label{SNSs_for_FT}
The measures $\nu_{\Phi,0}$ and $m_G$ have same null sets.
\end{cor}
\begin{rem}
By Corollary \ref{SNSs_for_FT} and \cite[Pg. 194, Theorem 4.14]{ORP}, the space $L^1(\nu_{\Phi,0})$ is the largest Banach function space with $\sigma$-order continuous norm such that $L^\Phi(G)$ is continuously embedded into $L^1(\nu_{\Phi,0})$ and the operator $\mathscr{F}_{\Phi,0}$ has a continuous extension with values in $c_0$-$\underset{[\pi]\in\widehat{G}}{\oplus}\mathcal{B}_2(\mathcal{H}_\pi).$ Further, this extension is unique and is the integration operator $I_{\nu_{\Phi,0}}.$
\end{rem}
\begin{cor}\label{Coincide}
The spaces $L^1(\nu_{\Phi,0})$ and $L^1(G)$ coincide. Further, the extended operator $I_{\nu_{\Phi,0}}$ coincides with the Fourier transform on $L^1(G).$
\end{cor}
\begin{proof}
By Lemma \ref{Variation_of_nu}, it follows that the spaces $L^1(G)$ and $L^1(|\nu_{\Phi,0}|)$ coincide. Further, the inclusions $$L^1(|\nu_{\Phi,0}|)\subseteq L^1(\nu_{\Phi,0})\subseteq L^1(|\langle \nu_{\Phi,0},\chi_{\{\pi_0\}} \rangle|)$$ are clear, where the first inclusion follows from \cite[Pg. 116, Lemma 3.14]{ORP}. Here $\pi_0$ and $\chi_{\{\pi_0\}}$ are defined as in proof of Lemma \ref{Variation_of_nu}. Also, it follows from the proof of Lemma \ref{Variation_of_nu} that $L^1(|\langle \nu_{\Phi,0},\chi_{\{\pi_0\}} \rangle|)$ coincides with $L^1(G).$ Thus the spaces $L^1(G)$ and $L^1(\nu_{\Phi,0})$ coincide. Hence, by the density of the space continuous functions on $G$ in both $L^1(G)$ and $L^\Phi(G),$ it follows that the extended operator is just the Fourier transform on $L^1(G).$
\end{proof}
\begin{cor}
Let $(\Phi,\Psi)$ be a pair of complementary Young functions satisfying the $\Delta_2$-condition. If $G$ is an infinite compact group then the domain of the optimal extension of $\mathscr{F}_{\Phi,0}$ contains $L^\Phi(G)$ as a proper subspace.
\end{cor}
\begin{proof}
Since $\Phi$ and $\Psi$ both satisfy the $\Delta_2$-condition, the space $L^\Phi(G)$ is reflexive and hence the proof follows from Corollary \ref{Coincide}.
\end{proof}

\subsection{Optimal extension of the Hausdorff-Young inequality}
Let $(\Phi,\Psi)$ denote a pair of complementary Young functions which are normalized. Further, we shall denote by $\ell^\Phi\mbox{-}\underset{[\pi]\in\widehat{G}}{\oplus}\mathcal{B}_2(\mathcal{H}_\pi),$ the Banach space consisting of $(x_{[\pi]})_{[\pi]\in\widehat{G}}\in\underset{[\pi]\in\widehat{G}}{\Pi}\mathcal{B}_2(\mathcal{H}_\pi)$ for which  $$\|(x_{[\pi]})\|_{\Phi}:=\inf\left\{ k>0|\underset{[\pi]\in\widehat{G}}{\sum} d_\pi^2\Phi\left( \frac{\|x_{[\pi]}\|_{\mathcal{B}_2(\mathcal{H}_\pi)}}{k\sqrt{d_{\pi}}} \right)\leq\Phi(1)\right\}<\infty.$$ The following is the Hausdorff-Young inequality for Orlicz spaces. See \cite[Pg. 211, Theorem 2]{RR2}. Also, see \cite{RR2} for any undefined notations regarding this theorem.
\begin{thm}
Let $(\Phi,\Psi)$ be a complementary pair of Young functions which are continuous and normalized. Suppose that
\begin{enumerate}[(i)]
\item $\Phi\prec'\Phi_0,$ where $\Phi_0(x)=c_1x^2$ and
\item $\Psi'(x)\leq c_2 x^r,$
\end{enumerate}
for some $r\geq 1,$ $c_1,c_2>0.$ Then the Fourier transform is a bounded linear operator from $L^\Phi(G)$ into $\ell^\Psi\mbox{-}\underset{[\pi]\in\widehat{G}}{\oplus}\mathcal{B}_2(\mathcal{H}_\pi).$
\end{thm}
Throughout this subsection, the complementary Young pair $(\Phi,\Psi)$ will satisfy the conditions of the above theorem. We shall denote by $\mathscr{F}_{\Phi,\Psi}$ the operator from $L^\Phi(G)$ to $\ell^\Psi\mbox{-}\underset{[\pi]\in\widehat{G}}{\oplus}\mathcal{B}_2(\mathcal{H}_\pi)$ given by the above theorem. Further, we shall also denote by $\nu_{\Phi,\Psi}$ the associated vector measure.
\begin{lem}\label{SNSs_for_HYI}
The measures $\nu_{\Phi,\Psi}$ and $m_G$ have same null sets.
\end{lem}
\begin{proof}
Let $A\in\mathcal{B}(G).$ Then $m_G(A)=0$ if and only if $\widehat{\chi_A}=0$ if and only if $\nu_{\Phi,\Psi}(A)=0.$ Hence the proof.
\end{proof}
\begin{rem}
By Lemma \ref{SNSs_for_HYI} and by \cite[Pg. 194, Theorem 4.14]{ORP}, the space $L^1(\nu_{\Phi,\Psi})$ is the largest Banach function space with $\sigma$-order continuous norm such that $L^\Phi(G)$ is continuously embedded into $L^1(\nu_{\Phi,\Psi})$ and the operator $\mathscr{F}_{\Phi,\Psi}$ has a continuous extension with values in $\ell^\Psi$-$\underset{[\pi]\in\widehat{G}}{\oplus}\mathcal{B}_2(\mathcal{H}_\pi).$ Further, this extension is unique and is the integration operator $I_{\nu_{\Phi,\Psi}}.$
\end{rem}
\begin{thm}
Let $(\Phi,\Psi)$ be a complementary pair of Young functions satisfying the $\Delta_2$-condition.
\begin{enumerate}[{\bf (i)}]
\item The space $L^\Phi(G)$ is a dense subspace of $L^1(\nu_{\Phi,\Psi})$ and $\|f\|_{\nu_{\Phi,\Psi}}\leq\|\mathscr{F}_{\Phi,\Psi}\|\|f\|_\Phi,$ $f\in L^\Phi(G).$
\item The space $L^1(\nu_{\Phi,\Psi})$ is a dense subspace of $L^1(G)$ and $\|f\|_1\leq\|f\|_{\nu_{\Phi,\Psi}},$ $f\in L^1({\nu_{\Phi,\Psi}}).$
\item The extended operator $I_{\nu_{\Phi,\Psi}}:L^1({\nu_{\Phi,\Psi}})\rightarrow\ell^\Psi$-$\underset{[\pi]\in\widehat{G}}{\oplus}\mathcal{B}_2(\mathcal{H}_\pi)$ is just the Fourier transform map, i.e., $I_{\nu_{\Phi,\Psi}}(f)=\widehat{f},$ for all $f\in L^1(\nu_{\Phi,\Psi}).$
\end{enumerate}
\end{thm}
\begin{proof}
{\bf (i)} Since $L^1(\nu_{\Phi,\Psi})$ is the optimal domain, it follows that $L^\Phi(G)$ is contained in $L^1(\nu_{\Phi,\Psi}).$ Further, by \cite[Pg. 185, Theorem 4.4]{ORP}, it follows that the inclusion of $L^\Phi(G)$ inside $L^1(\nu_{\Phi,\Psi})$ is continuous with its operator norm equal to $\|\mathscr{F}_{\Phi,\Psi}\|.$ The conclusion follows from the density of simple functions in both the spaces.

{\bf (ii)} Consider the element $\chi_{\{\pi_0\}}$ from $\ell^\Phi$-$\underset{[\pi]\in\widehat{G}}{\oplus}\mathcal{B}_2(\mathcal{H}_\pi).$ Here $\pi_0$ and $\chi_{\{\pi_0\}}$ are defined as in the proof of Lemma \ref{Variation_of_nu}. Here, note that $\chi_{\{\pi_0\}}\in B_{\ell^\Phi\mbox{-}\underset{[\pi]\in\widehat{G}}{\oplus}\mathcal{B}_2(\mathcal{H}_\pi)}.$ Again, as in the proof of Lemma \ref{Variation_of_nu}, one can show that $m_G=|\langle \nu_{\Phi,\Psi},\chi_{\{\pi_0\}} \rangle|.$ Thus, if $f\in L^1(\nu_{\Phi,\Psi}),$ then 
\begin{align*}
\int_G |f| dm_G =& \int_G |f| d|\langle \nu_{\Phi,\Psi},\chi_{\{\pi_0\}} \rangle|\leq \|f\|_{\nu_{\Phi,\Psi}}.
\end{align*}
Thus $f\in L^1(G).$ Further, the density of $L^1(\nu_{\Phi,\Psi})$ in $L^1(G),$ again, follows from the density of the simple functions.

{\bf (iii)} This is again a consequence of the density of simple functions and (ii).
\end{proof}

\section{Optimal extension of the convolution operator}

In this section, we consider the convolution operator and determine its optimal domain. Further, we also characterize when the extended operator will be compact. The results of this section are analogues of the results obtained in \cite{OR2}. Throughout this section, $(\Phi,\Psi)$ will denote a pair of complementary Young functions satisfying the $\Delta_2$-condition.

Let $\mu\in M(G).$ Define $T_{\Phi,\mu}:L^\Phi(G)\rightarrow L^\Phi(G)$ as $T_{\Phi,\mu}(f)=\mu*f,$ where $$\mu*f(t)=\int_G f(s^{-1}t)d\mu(s).$$ Then $T_{\Phi,\mu}$ is a bounded linear operator and $\|T_{\Phi,\mu}\|\leq \|\mu\|.$ Let $\nu_{\Phi,\mu}$ denote the associated vector measure. If $d\mu=gdm_G,$ then $\nu_{\Phi,g}$ will denote the measure $\nu_{\Phi,\mu}.$

\subsection{Optimal extension}
\begin{prop}\label{BI}
Let $g\in L^\Phi(G).$ Let $F_{\Phi,g}(t)=\tau_tg,\ t\in G,$ where $\tau_tg(s)=g(st).$
\begin{enumerate}[{\bf (i)}]
\item If $f\in L^1(G),$ then the function $t\mapsto fF_{\Phi,g}(t)$ from $G$ to $L^\Phi(G)$ is integrable (in the sense of Bochner).
\item If $f\in L^\Phi(G),$ then $g*f=\int_G fF_{\Phi,g} \ dm_G.$
\end{enumerate}
\end{prop}
\begin{proof}
{\bf (i)} Since $G$ is compact, it follows that the range of the continuous mapping $t\mapsto fF_{\Phi,g}(t)$ is compact and hence separable. Thus, by Pettis measurability theorem \cite[Pg. 42, Theorem 2]{DU}, it follows that the function $fF_{\Phi,g}$ is strongly measurable. By the fact that $L^\Phi(G)$ is translation invariant and that the translations are norm preserving, it follows that $\int_G\|fF_{\Phi,g}\|_{\Phi}\ dm_G<\infty,$ thus proving (i).

{\bf (ii)} By (i), Fubini's theorem and the fact that the dual of $L^\Phi(G)$ is $L^\Psi(G),$ we get that 
\begin{align*}
\langle g*f,h \rangle =&\left\langle \int_G fF_{\Phi,g} \ dm_G,h \right\rangle,\ \forall\ h\in L^\Psi(G). \qedhere
\end{align*}
\end{proof}
\begin{cor}\label{IRVMCOF}
Let $g\in L^\Phi(G).$ Then the $L^\Phi(G)$-valued measure $\nu_{\Phi,g}$ can be written as $$\nu_{\Phi,g}(A)=\int_A F_{\Phi,g}\ dm_G,$$ where $F_{\Phi,g}$ is defined as in Proposition \ref{BI}.
\end{cor}
\begin{proof}
Follows from (ii) of Proposition \ref{BI} by letting $f=\chi_A.$
\end{proof}
\begin{cor}\label{IRVMCOM}
Let $\mu\in M(G).$ If $f\in L^\Phi(G)$ such that $d\mu=fdm_G$ a.e., then there exists an integrable function (in the sense of Bochner) $F:G\rightarrow L^\Phi(G)$ such that for every Borel set $A$ of $G,$ $$\nu_{\Phi,\mu}(A)=\int_A F\ dm_G.$$
\end{cor}
\begin{proof}
By our assumption that the function $f$ is the Radon-Nikodym derivative of $\mu$ w.r.t. the Haar measure, it follows that $\nu_{\Phi,\mu}=\nu_{\Phi,f}.$ Now the proof follows from Corollary \ref{IRVMCOF}.
\end{proof}
The following theorem is the converse to the previous corollary.
\begin{thm}\label{RNTOS}
Let $\mu\in M(G).$ Suppose that an integrable function (in the sense of Bochner) $F:G\rightarrow L^\Phi(G)$ satisfies $\nu_{\Phi,\mu}(A)=\int_A F\ dm_G,$ $A\in\mathcal{B}(G).$ Then there exists $f\in L^\Phi(G)$ such that $d\mu=fdm_G$ a.e..
\end{thm}
\begin{proof}
Let $F:G\rightarrow L^\Phi(G)$ such that $\nu_{\Phi,\mu}(A)=\int_A F\ dm_G.$ By \cite[Pg. 42, Theorem 2]{DU}, there exists a Borel set $E$ of $G$ such that $m_G(E)=1$ and $\{F(x):x\in E\}$ is separable in $L^\Phi(G).$ So, without loss of generality, we can assume that $F(G)$ is a separable subset of $L^\Phi(G).$ Let $\mathscr{F}_{\Phi,0}$ denote the Fourier transform defined as in section 3. As $F(G)$ is separable and the map $\mathscr{F}_{\Phi,0}$ is continuous, $(\mathscr{F}_{\Phi,0}\circ F)(G)$ is a separable subset of $c_0$-$\underset{[\pi]\in\widehat{G}}{\oplus}\mathcal{B}_2(\mathcal{H}_\pi).$ Thus, by definition of the $c_0$ direct sum, there exists a countable subset $\Gamma$ of $\widehat{G}$ such that if $f\in F(G)$ then $\widehat{f}$ vanishes outside $\Gamma,$ i.e., \begin{equation}\label{VS}
\widehat{f}(\pi)=0_{d_\pi}\ \forall\ f\in F(G)\mbox{ and }\forall\ [\pi]\in\widehat{G}\setminus\Gamma.
\end{equation}

Let $[\pi]\in\widehat{G}.$ Now, consider the map $\widetilde{F}_\pi:G\rightarrow\mathcal{B}_2(\mathcal{H}_\pi)$ given by $\widetilde{F}_\pi(t)=\widehat{F(t)}(\pi).$ Since $G$ is compact, $L^\Phi(G)$ is continuously embedded in $L^1(G)$ and hence the above map is well-defined. Note that the linear maps $\mathscr{F}_{\Phi,0}:L^\Phi(G)\rightarrow c_0$-$\underset{[\pi]\in\widehat{G}}{\oplus}\mathcal{B}_2(\mathcal{H}_\pi)$ and $P_\pi:c_0$-$\underset{[\sigma]\in\widehat{G}}{\oplus}\mathcal{B}_2(\mathcal{H}_\sigma)\rightarrow\mathcal{B}_2(\mathcal{H}_\pi)$ are continuous, where $P_\pi$ is the projection onto $\mathcal{B}_2(\mathcal{H}_\pi)$ and hence, by \cite[Pg. 149]{ORP}, the composition $P_\pi\circ\mathscr{F}_{\Phi,0}\circ F$ is Bochner integrable. Further, as 
\begin{equation}\label{FComp}
\widetilde{F}_\pi=P_\pi\circ\mathscr{F}_{\Phi,0}\circ F,
\end{equation} 
it follows that $\widetilde{F}_\pi$ is Bochner integrable.

We now claim that, for every Borel subset $A$ of $G,$ 
\begin{equation}\label{claim}
\int_A\widetilde{F}_\pi(t)\ dm_G(t)=\int_A\widehat{\tau_t\mu}(\pi)\ dm_G(t),
\end{equation} where $\tau_t\mu(A)=\mu(At^{-1}).$ Let $A$ be a Borel subset of $G.$ Then 
\begin{align*}
\int_A\widetilde{F}_\pi(t)\ dm_G(t) =& \int_A (P_\pi\circ\mathscr{F}_{\Phi,0}\circ F)(t)\ dm_G(t)\ (\mbox{by }(\ref{FComp}))\\ =& P_\pi\circ\mathscr{F}_{\Phi,0}\left(\int_A  F(t)\ dm_G(t)\right)\ (P_\pi\mbox{ and }\mathscr{F}_{\Phi,0}\mbox{ are linear maps})\\ =& P_\pi\circ\mathscr{F}_{\Phi,0}\left(\nu_{\Phi,\mu}(A)\right) = \widehat{\nu_{\Phi,\mu}(A)}(\pi)\\ =& \widehat{\mu*\chi_A}(\pi) = \widehat{\chi_A}(\pi)\widehat{\mu}(\pi)\\ =& \left(\int_A\pi(t^{-1})\ dm_G(t)\right)\widehat{\mu}(\pi)\\ =& \int_A\pi(t^{-1})\widehat{\mu}(\pi)\ dm_G(t)\\ =& \int_A\int_G \pi(st)^*\ d\mu(s)\ dm_G(t)\\ =& \int_A\widehat{\tau_t\mu}(\pi)\ dm_G(t).
\end{align*}
Let $[\pi]\in\widehat{G}\setminus\Gamma.$ Then, by (\ref{VS}) and (\ref{claim}), we have, for every Borel set $A$ of $G,$ $$\int_A\pi(t^{-1})\widehat{\mu}(\pi)\ dm_G(t)=\int_A\widehat{\tau_t\mu}(\pi)\ dm_G(t)=\int_A\widetilde{F}_\pi(t)\ dm_G(t)=0_{d_\pi}.$$ This implies that $\pi(t^{-1})\widehat{\mu}(\pi)=0_{d_\pi}\ a.e.\ t\in G.$ As $\pi$ is a unitary representation, it follows that $\widehat{\mu}(\pi)=0_{d_\pi}$ and hence $\widehat{\tau_t\mu}(\pi)=0_{d_\pi},$ for every $t\in G.$ Thus, \begin{equation}\label{EGc}
\widehat{\tau_t\mu}(\pi)=\widetilde{F}_\pi(t)=0_{d_\pi}\ \forall\ t\in G.
\end{equation}
On the other hand, using the fact that $\Gamma$ is countable and (\ref{claim}), there exists a Borel set $A$ of $G$ such that $m_G(A)=1$ and \begin{equation}\label{EG}
\widehat{\tau_t\mu}(\pi)=\widetilde{F}_\pi(t)\ \forall\ t\in A\mbox{ and }\forall\ [\pi]\in\Gamma.
\end{equation}
By (\ref{EGc}) and (\ref{EG}) we have $$\widehat{\tau_t\mu}(\pi)=\widetilde{F}_\pi(t)=\widehat{F(t)}(\pi)\ \forall\ t\in A\mbox{ and }\forall\ [\pi]\in\widehat{G}$$ and therefore by the uniqueness of the Fourier transform it follows that $\tau_t\mu=F(t)\ \forall\ t\in A.$ Now the conclusion follows by applying $\tau_{t^{-1}}$ on both sides.
\end{proof}

\begin{lem}\label{Semivarion_Formula}
If $\mu\in M(G),$ then the semivariation of $\nu_{\Phi,\mu}$ satisfies the following:
for any Borel set $A,$ $$\frac{|\mu(G)|}{\Psi^{-1}(1)}m_G(A)\leq\|\nu_{\Phi,\mu}\|(A)=\underset{g\in B_{L^{\Psi}(G)}}{\sup}\int_A|\widetilde{\mu}*g|\ dm_G\leq \left[\Phi^{-1}\left(\frac{1}{m_G(A)}\right)\right]^{-1}\|\mu\|.$$
\end{lem}
\begin{proof}
Note that, for any $g\in L^\Psi(G),$ $T_{\Phi,\mu}^*(g)=\widetilde{\mu}*g.$ Let $A$ be a Borel subset of $G.$ For any $g\in L^\Psi(G),$ we have $$\langle \nu_{\Phi,\mu},g\rangle(A)=\langle T_{\Phi,\mu}(\chi_A),g\rangle=\langle \chi_A,T_{\Phi,\mu}^*(g)\rangle=\langle \chi_A,\widetilde{\mu}*g\rangle=\int_A\widetilde{\mu}*g\ dm_G.$$ Thus, by \cite[Pg. 46, Theorem 4]{DU} and from the definition of the semivariation of a vector measure, it follows that $$\|\nu_{\Phi,\mu}\|(A)=\underset{g\in B_{L^{\Psi}(G)}}{\sup}\int_A|\widetilde{\mu}*g|\ dm_G.$$
Note that, by \cite[Pg. 78, Corollary 7]{RR}, it follows that $\left\|\frac{1}{\Psi^{-1}(1)}\chi_G\right\|_{\Psi}^o=1.$ Thus, from the above equality, it follows that $$\|\nu_{\Phi,\mu}\|(A)\geq\frac{1}{\Psi^{-1}(1)}\int_A|\widetilde{\mu}*\chi_G|\ dm_G.$$ Now the inequality on the left follows from the fact that $\widetilde{\mu}*\chi_G=\mu(G).$ Further, the right inequality follows from the H\"older's inequality for Orlicz spaces \cite[Pg. 62, Remark 1]{RR} and \cite[Pg. 78, Corollary 7]{RR}.
\end{proof}
\begin{cor}\label{VMACHM}
If $\mu\in M(G),$ then the vector measure $\nu_{\Phi,\mu}$ is absolutely continuous w.r.t. $m_G.$ Further, if $\mu(G)$ is non-zero, then $m_G$ is absolutely continuous w.r.t. $\nu_{\Phi,\mu}.$
\end{cor}
\begin{proof}
The proof of this follows from Lemma \ref{Semivarion_Formula}.
\end{proof}
\begin{rem}
If $\mu\in M(G)$ is such that $\mu(G)\neq 0$ then, by Corollary \ref{VMACHM}, the measures $\nu_{\Phi,\mu}$ and $m_G$ have same null sets. Hence, by \cite[Pg. 194, Theorem 4.14]{ORP}, the space $L^1(\nu_{\Phi,\mu})$ is the largest Banach function space with $\sigma$-order continuous norm such that $L^\Phi(G)$ is continuously embedded into $L^1(\nu_{\Phi,\mu})$ and the operator $T_{\Phi,\mu}$ has a continuous extension with values in $L^\Phi(G).$ Further, this extension is unique and is the integration operator $I_{\nu_{\Phi,\mu}}.$
\end{rem}
\begin{cor}\label{ECFV}
Let $\mu\in M(G).$ Then, there exists $f\in L^\Phi(G)$ such that $d\mu=fdm_G$ a.e. if and only if $|\nu_{\Phi,\mu}|<\infty.$
\end{cor}
\begin{proof}
Forward part follows from Corollary \ref{IRVMCOM} and \cite[Pg. 46, Theorem 4(iv)]{DU}. For the converse, note that the space $L^\Phi(G)$ is reflexive and hence, by \cite[Pg. 76, Corollary 13]{DU}, $L^\Phi(G)$ has the Radon-Nikodym property. Thus the conclusion follows from Corollary \ref{VMACHM} and Theorem \ref{RNTOS}.
\end{proof}

\subsection{Properties of the optimal domain}
From now onwards $\mu$ will denote a complex Radon measure such that $\mu(G)\neq 0.$

\begin{prop}
Let $f$ be a  complex-valued Borel measurable function on $G.$ Then, $f\in L^1(\nu_{\Phi,\mu})$ if and only if $\int_G|f||\widetilde{\mu}*g|\ dm_G<\infty$ for every $g\in L^\Psi(G).$ Further, if $f\in L^1(\nu_{\Phi,\mu}),$ then \begin{equation}\label{L1norm}
\|f\|_{\nu_{\Phi,\mu}}=\underset{g\in B_{L^\Psi(G)}}{\sup}\int_G|f||\widetilde{\mu}*g|\ dm_G.
\end{equation}
\end{prop}
\begin{proof}
Since $\Phi$ satisfies the $\Delta_2$-condition, by \cite[Pg. 139, Theorem 5]{RR}, $L^\Phi(G)$ does not contain a copy of $c_0$ and hence, by \cite[Pg. 138]{ORP}, the spaces $L^1(\nu_{\Phi,\mu})$ and $L^1_w(\nu_{\Phi,\mu})$ coincide. Let $g\in L^\Psi(G).$ Then, by following the steps as in Lemma \ref{Semivarion_Formula} and by \cite[Pg. 46, Theorem 4]{DU}, it follows that \begin{equation}\label{Variation_Formula}
|\langle \nu_{\Phi,\mu},g\rangle|(A)=\int_A|\widetilde{\mu}*g|\ dm_G.
\end{equation} Thus, $f\in L^1(\nu_{\Phi,\mu})$ if and only if $f\in L^1_w(\nu_{\Phi,\mu}),$ i.e., $\int_G|f|\ d|\langle \nu_{\Phi,\mu},g\rangle|<\infty$ for every $g\in L^\Psi(G),$ i.e., $\int_G|f||\widetilde{\mu}*g|\ dm_G<\infty$ for every $g\in L^\Psi(G).$ Now the formula for the norm of $f\in L^1(\nu_{\Phi,\mu})$ follows from the definition and from (\ref{Variation_Formula}).
\end{proof}
\begin{thm}\label{LPhi2L1}\mbox{ }
\begin{enumerate}[{\bf (i)}]
\item If $f\in L^\Phi(G),$ then \begin{equation*}\label{LPhiL1}
\|f\|_{\nu_{\Phi,\mu}}\leq\|f\|_\Phi\|\mu\|.
\end{equation*}
\item The inclusion $\imath:L^\Phi(G)\rightarrow L^1(\nu_{\Phi,\mu})$ is continuous with the operator norm satisfying $$|\mu(G)|\leq\|\imath\|\leq\|\mu\|.$$
\end{enumerate}
\end{thm}
\begin{proof}
{\bf (i)} The proof follows from (\ref{L1norm}) and from the H\"older's inequality for Orlicz spaces \cite[Pg. 62, Remark 1]{RR}. 

{\bf (ii)} The continuity of the inclusion map $\imath$ and the inequality on the right follows from (i). We now prove the inequality on the left side. By (\ref{L1norm}), it is clear that $$\|\imath\|=\underset{g\in B_{L^\Psi(G)}}{\sup}\|\widetilde{\mu}*g\|_{\Psi}^o$$ and hence $\|\imath\|\geq\|\widetilde{\mu}*g\|_{\Psi}^o$ for all $g\in B_{L^\Psi(G)}.$ In particular, taking $g=\frac{1}{\Psi^{-1}(1)}\chi_G,$ we get the required inequality.
\end{proof}
\begin{cor}\label{dense}
The space $L^\Phi(G)$ is dense in $L^1(\nu_{\Phi,\mu}).$
\end{cor}
\begin{proof}
This follows from the previous theorem and the fact that the space of simple functions is dense in the both of the spaces w.r.t. their respective norms.
\end{proof}
\begin{cor}\label{EO}
The extended map $I_{\nu_{\Phi,\mu}}:L^1(\nu_{\Phi,\mu})\rightarrow L^\Phi(G)$ is given by $I_{\nu_{\Phi,\mu}}(f)=\mu*f.$
\end{cor}
\begin{proof}
This follows from \cite[Pg. 185, Theorem 4.4]{ORP} and Corollary \ref{dense}.
\end{proof}
\begin{thm}\label{inclusion}\mbox{ }
\begin{enumerate}[{\bf (i)}]
\item If $f\in L^1(\nu_{\Phi,\mu}),$ then \begin{equation*}
\|f\|_1\leq\frac{\Psi^{-1}(1)}{|\mu(G)|}\|f\|_{\nu_{\Phi,\mu}}.
\end{equation*}
\item The inclusion $\jmath:L^1(\nu_{\Phi,\mu})\rightarrow L^1(G)$ is continuous with the operator norm satisfying $$\frac{1}{\|\nu_{\Phi,\mu}\|}\leq\|\jmath\|\leq\frac{\Psi^{-1}(1)}{|\mu(G)|}.$$
\item The range of $\jmath$ is a dense and translation invariant subspace of $L^1(G).$
\end{enumerate}
\end{thm}
\begin{proof}
{\bf (i)} It follows from (\ref{L1norm}) that $\|f\|_{\nu_{\Phi,\mu}}\geq\int_G|f||\widetilde{\mu}*g|\ dm_G$ for all $g\in B_{L^\Psi(G)}.$ In particular, taking $g=\frac{1}{\Psi^{-1}(1)}\chi_G,$ we get the required inequality.

{\bf (ii)} The continuity of the inclusion map $\jmath$ and the inequality on the right follows from (i). We now prove the inequality on the left side. From the definition of the operator norm, it follows that $\|\jmath\|\geq \|\jmath(f)\|_1$ for all $f\in B_{L^1(\nu_{\Phi,\mu})}.$ In particular, taking $f=\frac{1}{\|\nu_{\Phi,\mu}\|}\chi_G,$ we get the required inequality.

{\bf (iii)} Denseness of the range of $\jmath$ follows from the fact that the space of simple functions is dense in the both of the spaces w.r.t. their respective norms. Further, translation invariance follows from (\ref{L1norm}) and the fact that translation on $L^\Psi(G)$ is an isometry.
\end{proof}

\subsection{Compactness of the extended operator}
A function $f:G\rightarrow X$ is said to be Pettis integrable if $\langle f,x^\prime\rangle\in L^1(G),~x^\prime\in X^\prime$ and if for each $A\in\mathfrak{B}(G)$ there exists a unique $x_A\in X$ such that $\int_A\langle f,x^\prime\rangle\,dm_G=\langle x_A,x^\prime\rangle,~x^\prime\in X^\prime.$

We now characterize the compactness of the extended operator in terms of the measure $\nu_{\Phi,\mu}.$
\begin{thm}\label{ECC}
Let $\mu\in M(G).$ Then the following are equivalent:
\begin{enumerate}[{\bf (i)}]
\item The $L^\Phi(G)$-valued integral operator $I_{\nu_{\Phi,\mu}}$ is compact.
\item The measure $\nu_{\Phi,\mu}$ has finite variation.
\item  There exists a $L^\Phi(G)$-valued Bochner integrable function $F_1$ such that for every Borel set $A$ of $G,$ $$\nu_{\Phi,\mu}(A)=\int_AF_1\ dm_G.$$
\item There exists a $L^\Phi(G)$-valued Pettis integrable function $F_2$ such that for every Borel set $A$ of $G,$ $$\nu_{\Phi,\mu}(A)=\int_AF_2\ dm_G.$$
\item There exists a Borel set $A$ of $G$ such that $|\nu_{\Phi,\mu}|(A)$ is non-zero and finite.
\item There exists $f\in L^\Phi(G)$ such that $d\mu=fdm_G.$
\end{enumerate}
\end{thm}
\begin{proof}
(i)$\Rightarrow$(ii) is a consequence of \cite[Theorem 4]{ORR}. 

(ii)$\Rightarrow$(iii) follows from Corollary \ref{ECFV} and Corollary \ref{IRVMCOM}.

(iii)$\Rightarrow$(iv) follows from the fact that a Bochner integrable function is Pettis integrable \cite[Pg. 149]{ORP}.

(iv)$\Rightarrow$(v) is a consequence of the fact that a vector measure defined by using a Pettis integrable function has $\sigma$-finite variation \cite[Corollary 1]{M}.

(v)$\Rightarrow$(vi). Let $A$ be a Borel subset subset of $G$ such that $0<|\nu_{\Phi,\mu}|(A)<\infty.$ Then, by \cite[Pg. 106]{ORP} and Lemma \ref{Semivarion_Formula}, $m_G(A)>0.$ As $L^\Phi(G)$ is reflexive, by \cite[Pg. 76, Corollary 13]{DU} applied to the measure $\nu_{\Phi,\mu}$ restricted to $A,$ we get a function $F:A\rightarrow L^\Phi(G)$ which is the Radon-Nikodym derivative of $\nu_{\Phi,\mu}$ restricted to $A$ w.r.t. the Haar measure $m_G.$ Now the argument given for Theorem \ref{RNTOS} can be modified to obtain the necessary conclusion.

(vi)$\Rightarrow$(i). By Corollary \ref{ECFV} the measure $\nu_{\Phi,\mu}$ has finite variation and by \cite[Pg. 46, Theorem 4]{DU} and Corollary \ref{IRVMCOF}, it follows that the variation of $\nu_{\Phi,\mu}$ is $\|f\|_\Phi m_G.$ Further, the Radon-Nikodym derivative $\frac{d\nu_{\Phi,\mu}}{d|\nu_{\Phi,\mu}|}$ is $\frac{F_{\Phi,f}}{\|f\|_\Phi}$ with a compact range. Thus the compactness of the extended operator $I_{\nu_{\Phi,\mu}}$ follows from \cite[Theorem 1]{ORR}.
\end{proof}
Our next characterization is in terms of the $L^1$-spaces.
\begin{cor}\label{ECCL1}
Let $\mu\in M(G).$ Then the following are equivalent:
\begin{enumerate}[{\bf (i)}]
\item The $L^\Phi(G)$-valued integral operator $I_{\nu_{\Phi,\mu}}$ is compact.
\item The spaces $L^1(G)$ and $L^1(|\nu_{\Phi,\mu}|)$ coincide.
\item The spaces $L^1(G)$ and $L^1(\nu_{\Phi,\mu})$ coincide.
\item The spaces $L^1(\nu_{\Phi,\mu})$ and $L^1(|\nu_{\Phi,\mu}|)$ coincide.
\end{enumerate}
\end{cor}
\begin{proof}
(i)$\Rightarrow$(iv) follows from \cite[Theorem 1]{ORR}. Conversely, (iv) implies that $|\nu_{\Phi,\mu}|$ is finite as the constant function $1\in L^1(\nu_{\Phi,\mu}).$ Hence, by Theorem \ref{ECC}, (i) follows.

(i)$\Rightarrow$(ii). By Theorem \ref{ECC} there exists $f\in L^\Phi(G)$ such that $d\mu=fdm_G.$ Further, by \cite[Pg. 46, Theorem 4]{DU} and Corollary \ref{IRVMCOF}, it follows that the variation of $\nu_{\Phi,\mu}$ is $\|f\|_\Phi m_G.$ Hence (ii).

(ii)$\Rightarrow$(iii) follows from \cite[Pg. 116, Lemma 3.14]{ORP} and Theorem \ref{inclusion}. We now prove (iii)$\Rightarrow$(i). Let $f\in L^1(G).$ By our assumption, $f\in L^1(\nu_{\Phi,\mu})$ and hence, by Corollary \ref{EO}, $\mu*f\in L^\Phi(G).$ Note that, by \cite[Pg. 152]{ORP}, $$\underset{f\in B_{L^1(G)}}{\sup}\|\mu*f\|_\Phi=\underset{f\in B_{L^1(\nu_{\Phi,\mu})}}{\sup}\|I_{\nu_{\Phi,\mu}}(f)\|=\|I_{\nu_{\Phi,\mu}}\|=1.$$ Now the conclusion follows from Lemma \ref{RNDOSF} and Theorem \ref{ECC}.
\end{proof}
As our final result, we show that the optimal extension is genuine.
\begin{cor}
Let $G$ be an infinite compact group and let $\mu\in M(G)$ be such that $d\mu=fdm_G$ for some $f\in L^\Phi(G).$ Then the optimal extension of $T_{\Phi,\mu}$ is proper, i.e., $L^\Phi(G)$ is a proper subspace of $L^1(\nu_{\Phi,\mu}).$
\end{cor}
\begin{proof}
By Theorem \ref{LPhi2L1}, $L^\Phi(G)\subseteq L^1(\nu_{\Phi,\mu}).$ Further, our assumption implies, by Theorem \ref{ECC} and Corollary \ref{ECCL1}, that $L^\Phi(G)\subseteq L^1(G).$ As $\Phi$ and $\Psi$ satisfy the $\Delta_2$ condition, $L^\Phi(G)$ is reflexive and as $G$ is infinite it follows that $L^\Phi(G)$ is a proper subspace of $L^1(G)$ and hence a proper subspace of $L^1(\nu_{\Phi,\mu}).$
\end{proof}

\section*{Acknowledgement}
The first author would like to thank the University Grants Commission, India, for providing the research grant.

\end{document}